\documentclass[11pt,reqno]{amsart}

\newtheorem{proposition}{Proposition}[section]
\newtheorem{lemma}[proposition]{Lemma}

\newtheorem{theorem}[proposition]{Theorem}

\theoremstyle{definition}
\newtheorem{definition}[proposition]{Definition}

\newtheorem{examples}[proposition]{Examples}
\newtheorem{remark}[proposition]{Remark}

\newcommand{\thlabel}[1]{\label{th:#1}}
\newcommand{\thref}[1]{Theorem~\ref{th:#1}}
\newcommand{\selabel}[1]{\label{se:#1}}

\newcommand{\delabel}[1]{\label{de:#1}}

\newcommand{\eqlabel}[1]{\label{eq:#1}}
\newcommand{\equref}[1]{(\ref{eq:#1})}

\newcommand{\Aut}{{\rm Aut}\,}

\newcommand{\im}{{\rm Im}\,}

\def\NN{{\mathbb N}}

\def\ZZ{{\mathbb Z}}

\newcommand{\Cc}{\mathcal{C}}

\def\*C{{}^*\hspace*{-1pt}{\Cc}}
\def\text#1{{\rm {\rm #1}}}

\usepackage{amssymb}

\usepackage [UglyObsolete,tight,heads=LaTeX]{diagrams}

\usepackage{color,amssymb,graphicx}

\input xy
\xyoption {all} \CompileMatrices

\begin{document}

\title[Crossed product of cyclic groups]
{Crossed product of cyclic groups}

\author{Ana-Loredana Agore}\thanks{The authors where supported by
CNCSIS grant 24/28.09.07 of PN II "Groups, quantum groups, corings
and representation theory".}
\author{Drago\c{s} Fr\u{a}\c{t}il\u{a}}

\address{Faculty of Mathematics and Computer Science,
University of Bucharest, Str. Academiei 14, RO-010014 Bucharest 1,
Romania} \email{ana.agore@fmi.unibuc.ro}
\email{dragos.fratila@fmi.unibuc.ro}



\begin{abstract}{All crossed products of two cyclic groups are
explicitly described using generators and relations. A necessary
and sufficient condition for an extension of a group by a group to
be a cyclic group is given.}
\end{abstract}

\maketitle

\section*{Introduction}

One of the most frequently used results in elementary number
theory is the famous ancient Chinese Remainder Theorem. The
Chinese Remainder theorem can be restated in an abstract and
elegant language of group theory as follows: the direct product
$H\times G$ of two groups is a cyclic group iff the groups are
finite, cyclic of coprime orders. The direct product $H \times G$
is the trivial example of an extension of a group $H$ by a group
$G$, that is there exists an exact sequence of groups:
$$
\begin{diagram}
1 & \rTo & H & \rTo^{i_{H}} & H \times G & \rTo^{\pi_{G}} & G & \rTo
& 1
\end{diagram}
$$

It is therefore natural and tempting to consider the most general
problem:

\emph{\textbf{Problem 1}: Let $(E, i, \pi)$ be an
extension of $H$ by $G$: i.e. $E$ is a group, $i : H \to E$ and
$\pi : E \to G$ are morphisms of groups such that the sequence
$$
\begin{diagram}
1 &\rTo & H & \rTo^{i} & E & \rTo^{\pi} & G & \rTo & 1
\end{diagram}
$$
is exact. Give a necessary and sufficient condition for the group
$E$ to be cyclic.}

The main theorem of the paper (\thref{main_theorem}) gives a
complete answer to the above question. From this point of view
\thref{main_theorem} can be considered as an interesting and
non-trivial generalization of the Chinese Remainder Theorem.

To obtain this result we will go through the following steps: we
will do a survey of the famous "extension problem" of H\"{o}lder
\cite{holder}, then we will work in an equivalent way with the
crossed systems instead of exact sequences and in the end we will
explicitly compute all the symmetric, normalized 2-cocycles for
two fixed cyclic groups and the (crossed)twisted products
associated.

The extension problem was first stated by H\"{o}lder \cite{holder}. A recent survey and new results
related to the extension problem are obtained in \cite{MA}. In particular, crossed products arise
naturally when dealing with group extensions. \cite[Corollary 1.8]{MA} is another formulation of
Schreier theorem and shows that the existence of an extension of $H$ by $G$ is equivalent to the
existence of a normalized crossed system $(H, G, \alpha, f)$, where $\alpha : G \rightarrow \Aut
(H)$ is a weak action and $f : G \times G \rightarrow H$ is an $\alpha$-cocycle. The classical
extension problem of H\"{o}lder was restated in \cite[Problem 1]{MA} in a computational manner as
follows:

\emph{\textbf{Problem 2}: Let $H$ and $G$ be two fixed groups.
Describe all normalized crossed systems $(H, G, \alpha, f)$ and
classify up to isomorphism all crossed products $H \#_{\alpha}^{f}
\, G$.}

The first notable result regarding the extension problem was given
by O. L. H\"{o}lder (\thref{thmHolder}), who uses generators and
relations to describe all crossed products of two finite cyclic
groups. In the section $2$ of the paper we complete the structure
and we shall describe all crossed products of two cyclic groups
(not necesary finite) using generators and relations: see
\thref{finitvsinfinit}, \thref{infinitvsfinit} and
\thref{cicliceinfinite}. Related to Problem 2 another question
arise:

\emph{\textbf{Problem 3}: Let $\Lambda$ be a class of groups. Find
necessary and sufficient conditions for $(H, G, \alpha, f)$ such
that the crossed product $H \#_{\alpha}^{f} \, G$ belongs to
$\Lambda$.}

In \cite[Corollary 1.15]{MA} a complete answer is given for the
above problem in the case of abelian groups: the crossed product
$H \#_{\alpha}^{f} \, G$ is an abelian group if and only if $H$
and $G$ are abelian groups, $\alpha$ is the trivial action and $f$
is a symmetric 2-cocycle.

The present paper deals with this problem in the case of cyclic
groups. In the first section we recall the construction and
fundamental properties of crossed product of groups. In section
$2$ we describe crossed products between all types of cyclic
groups. Using the aforementioned results, in Section $3$ we find
necessary and sufficient conditions for a crossed product to be a
cyclic group (\thref{main_theorem}) which is the main result of
this paper.

\section{Preliminaries}\selabel{1}
Let us fix the notations that will be used throughout the paper.
$C_n$ will be a cyclic group of order $n$ generated by $a$: $C_n =
\{1\,,\, a\,,\, a^2\,,\, \cdots ,\, a^{n-1} \}$ and
$C_{g}=\{g^{k}\,|\,k\in \ZZ\}$ will denote a cyclic infinite
group. Let $H$ and $G$ be two groups. $\Aut (H)$ denotes the group
of automorphisms of a group $H$ and $Z(H)$ the center of $H$. A
map $f : G \times G \rightarrow H$ is called \emph{symmetric} if
$f (g_1, g_2) = f (g_2, g_1)$ for any $g_1$, $g_2 \in G$. For a
map $\alpha : G \rightarrow \Aut (H)$ we shall use the notation
$$
\alpha (g) (h) = g\triangleright h
$$
for all $g\in G$ and $h\in H$.

The maps $\alpha$ and $f$ are called \emph{trivial} if
$g\triangleright h = h$ for all $g\in G$ and $h\in H$,
respectively $f(g_{1},g_{2})=1$ for all $g_{1}, g_{2} \in G$.

\begin{definition}\delabel{crossedsystem}
A \textit{crossed system} of groups is a quadruple $(H, G, \alpha,
f)$, where $H$ and $G$ are two groups, $\alpha : G \rightarrow
\Aut (H)$ and $f : G \times G \rightarrow H$ are two maps such
that the following compatibility conditions hold:
\begin{eqnarray}
g_1 \triangleright (g_2\triangleright h) &=& f(g_1, g_2) \,
\bigl((g_1g_2)\triangleright  h \bigl)\, f(g_1, g_2)^{-1}
\eqlabel{WA} \\
f(g_1,\, g_2)\, f(g_1 g_2, \, g_3) &=&  \bigl(g_1 \triangleright
f(g_2, \, g_3) \bigl) \, f(g_1, \, g_2g_3) \eqlabel{CC}
\end{eqnarray}
for all $g_1$, $g_2$, $g_3 \in G$. The crossed system $\Gamma =
(H, G, \alpha, f)$ is called \textit{normalized} if $f(1, 1) = 1$.
The map $\alpha : G \rightarrow \Aut (H)$ is called a \textit{weak
action} and $f : G \times G \rightarrow H$ is called an
$\alpha$-\textit{cocycle}.
\end{definition}

If $(H, G, \alpha, f)$ is a normalized crossed system then
\cite[Lemma 1.2]{MA}

\begin{equation}\eqlabel{norm2}
f (1, g) = f (g, 1) = 1 \qquad {\rm and} \qquad 1 \triangleright h
= h
\end{equation}
for any $g \in G$ and $h\in H$.

Let $H$ and $G$ be groups, $\alpha : G \rightarrow \Aut (H)$ and
$f : G \times G \rightarrow H$ two maps. Let $H \#_{\alpha}^{f} \,
G : = H\times G$ as a set with a binary operation defined by the
formula:
\begin{equation}\eqlabel{4}
(h_1,\, g_1)\cdot (h_2, \, g_2) : = \bigl( h_1 (g_1 \triangleright
h_2) f(g_1, \, g_2), \, g_1g_2 \bigl)
\end{equation}
for all $h_1$, $h_2 \in H$, $g_1$, $g_2 \in G$. Then \cite[Theorem
1.3]{MA} $\bigl ( H \#_{\alpha}^{f} \, G, \, \cdot \bigl)$ is a
group with the unit $1_{H \#_{\alpha}^{f} \, G} = \bigl( 1 , \,
1\bigl)$ if and only if $(H, G, \alpha, f)$ is a normalized
crossed system. In this case the group $H \#_{\alpha}^{f} \, G$ is
called the \textit{crossed product of $H$ and $G$} associated to
the crossed system $(H, G, \alpha, f)$.

The following \cite[Examples 1.5]{MA} are basic examples of
special cases of a crossed product of two groups.

\begin{examples}
1. Let $H$ and $G$ be two groups and $\alpha$, $f$ be the trivial
maps. Then $\Gamma = (H, G, \alpha, f)$ is a crossed system called
the \textit{trivial crossed system}. The crossed product $H
\#_{\alpha}^{f} \, G = H\times G$ is the direct product of $H$ and
$G$.
\newline 2. Let $H$ and $G$ be two groups and $f: G\times G \rightarrow H$ the
trivial map. Then $(H, G, \alpha, f)$ is a crossed system if and
only if $\alpha : G \rightarrow \Aut (H)$ is a morphism of groups.
In this case the crossed product $H \#_{\alpha}^{f} \, G = H
\ltimes_{\alpha} G$, the semidirect product of $H$ and $G$.
\newline 3. Let $H$ and $G$ be two groups and $\alpha : G
\rightarrow \Aut (H)$ the trivial action. Then $(H, G, \alpha, f)$
is a crossed system if and only if $\im (f) \subseteq Z(H)$ and
\begin{equation}\eqlabel{CCpropiu}
f(g_1, \, g_2) f(g_1g_2, \, g_3) = f(g_2, \, g_3) f(g_1, \,
g_2g_3)
\end{equation}
for all $g_1$, $g_2$, $g_3 \in G$, that is $f : G\times G
\rightarrow Z(H)$ is a $2$-cocycle. The crossed product $H
\#_{\alpha}^{f} \, G$ associated to this crossed system will be
denoted by $H\times^{f} \, G$ and was called in \cite{MA} the
\textit{twisted product} of $H$ and $G$ associated to the
$2$-cocycle $f: G\times G \rightarrow Z(H)$. Explicitly, the
multiplication of a twisted product of groups $H\times^{f} \, G$
is given by the formula:
\begin{equation}\eqlabel{tw4}
(h_1,\, g_1)\cdot (h_2, \, g_2) : = \bigl( h_1 h_2 f(g_1, \, g_2),
\, g_1g_2 \bigl)
\end{equation}
for all $h_1$, $h_2 \in H$, $g_1$, $g_2 \in G$.
\end{examples}

The next well known theorem is the main application of the crossed
product construction: it is a reconstruction theorem of a group
from a normal subgroup and the quotient.

\begin{theorem}\thlabel{4}
Let $E$ be a group $H\unlhd E$ be a normal subgroup of $E$ and $G
: = E/ H$ be the quotient of $E$ by $H$. Then there exists two
maps $\alpha : G \rightarrow \Aut (H)$ and $f : G \times G
\rightarrow H$ such that $(H, G, \alpha, f)$ is a normalized
crossed system and $ E \cong H \#_{\alpha}^{f} \, G$ (isomorphism
of groups).
\end{theorem}
For complete proofs and further details we refer to \cite{agore},
\cite[Theorem 1.6]{MA} or \cite{Weibel}.

\hspace{1cm}
\section{\textbf{Crossed product of cyclic groups}}\selabel{2}
\hspace{1cm}

Our purpose in this section is to describe using generators and
relations all crossed products between cyclic groups. As mentioned
in the introduction, the first important result in literature for
the first part of the extension problem was proved by H\"{o}lder
himself \cite[Theorem 12.9]{grillet}. It describes the crossed
product of two finite cyclic groups: for the sake of completeness
we present bellow a short proof of this theorem.

\begin{theorem}\thlabel{thmHolder}
\textbf{(Holder)} A finite group $E$ is isomorphic to a crossed
product $C_n \#_{\alpha}^{f} \, C_m$ if and only if $E$ is the
group generated by two generators $a$ and $b$ subject to the
relations
\begin{equation}\eqlabel{relatiicicice1}
a^n = 1, \qquad b^m = a^i, \qquad b^{-1} a b = a^j
\end{equation}
where $i$, $j \in \{0\,,\, 1, \cdots ,\, n-1 \}$ such that
\begin{equation}\eqlabel{relatiicicice2}
i(j-1) \equiv 0 ({\rm mod} \, n), \qquad j^m \equiv 1 ({\rm mod}
\, n)
\end{equation}
We denote this group by $C_n \#_{i}^{j} \, C_m$.
\end{theorem}

\begin{proof}
Assume first that E is isomorphic to a crossed product $C_n
\#_{\alpha}^{f} \, C_m$. Hence $C_n\unlhd E$ and $E/C_n \simeq
C_m$. It follows that $C_{n}=<a\, |\, a^{n}=1> \unlhd  E$ and
there exists $b\in E$ such that $E/ C_n = \{C_n\,,\, bC_n\,,...,
\,b^{m-1}C_n\}$ and $b^{m}\in C_n$. That is, there exists $i\in
\{0\,,\,1\,,...,\, n-1\}$ such that:
\begin{equation}\eqlabel{ci1}
b^{m}=a^{i}
\end{equation}
Since $C_n\unlhd E$ we obtain that $b^{-1}ab \in C_n$ and so there
exists $j \in \{0\,,\,1\,..., \,n-1\}$ such that:
\begin{equation}\eqlabel{ci2}
b^{-1}ab=a^{j}
\end{equation}
A direct computation shows that:
$$
b^{-1}a^{i}b\stackrel{\equref{ci1}}
{=}b^{-1}b^{m}b=b^{m}\stackrel{\equref{ci1}} {=}a^{i} \quad {\rm
and} \quad b^{-1}a^{i}b\stackrel{\equref{ci2}} {=}a^{ij}
$$
It follows from here that $a^{i(j-1)}=1$ and so $i(j-1) \equiv 0
({\rm mod} \, n)$. In a similar way we obtain:
$$
b^{-m}ab^{m}\stackrel{\equref{ci1}} {=}a^{-i}aa^{i}=a \quad {\rm
and} \quad a^{j^{2}}\stackrel{\equref{ci2}}
{=}(b^{-1}ab)^{j}=b^{-1}a^{j}b\stackrel{\equref{ci2}}
{=}b^{2}ab^{2}
$$
and by induction : $b^{-m}ab^{m}=a^{j^{m}}$. Hence $a=a^{j^{m}}$,
that is $j^m \equiv 1 ({\rm mod} \, n)$.

Conversely, assume that relations \equref{relatiicicice1} and
\equref{relatiicicice2} hold. We need to show that $C_n\unlhd E$,
that is $xa^{t}x^{-1} \in C_n$ for every $x \in E$ and $t \in
\{0\,,\,1\,,...,\, n-1\}$. Since $x \in E$ we have
$x=x_1x_2...x_k$ where $k \in \NN$, $x_s\in
\{a\,,\,b\,,\,a^{-1}\,,\,b^{-1}\}$ and $s \in
\{0\,,\,1\,,...,\,k\}$. We obtain that
$ga^{t}g^{-1}=x_1x_2...x_ka^{t}(x_k)^{-1}...(x_1)^{-1}$. It is
easy to see by a direct computation that $x_ka^{t}(x_k)^{-1} \in
C_n$ for every $x_k \in \{a\,,\,b\,,\,a^{-1}\,,\,b^{-1}\}$ and so,
by induction it follows that $ga^{t}g^{-1} \in C_n$. Hence
$C_n\unlhd E$. In a similar way, it can be showed that every
element of the group E can be written as $a^{p}b^{q}$ for $p,q \in
\ZZ$. Hence $|E|=mn$ and so $|E/C_n|=m$, $E/C_n=\{C_n\,,\,C_n
b\,,...,\,C_n b^{m-1}\}$ that is, the group $E$ has a normal
subgroup $C_n$ and $E/C_n\simeq C_m$. By \thref{4}, there exists
$(C_n,C_m,\alpha, f)$ a crossed system such that $E\simeq C_n
\#_{\alpha}^{f} \, C_m$.
\end{proof}

\begin{theorem}\thlabel{finitvsinfinit}
A group E is isomorphic to a crossed product $C_{n}
\#_{\alpha}^{f} \, C_{g}$ if and only if there exists $t\in\ZZ,
(t,n)=1$ such that $E\simeq <a, \, g \, | \, a^{n}=1, \, g^{-1}ag
= a^{t}>$.
\end{theorem}
\begin{proof}
Suppose first that $E\simeq C_{n} \#_{\alpha}^{f} \, C_{g}$. Hence
$C_{n}\unlhd E$ and $E/C_{n}\simeq C_{g}$. That is
$E/C_{n}=\{g^{k}C_{n} \,| \,k \in \ZZ\}$. Since $C_{n}\unlhd E$ we
obtain that $C_{n}=<a \,| \,a^{n}=1>\subseteq E$ and $g^{-1}ag \in
C_{n}$. That is, there exists $t \in \{0 \,,1 \,,...,n-1\}$ such
that
\begin{equation}\eqlabel{fvsinf}
g^{-1}ag=a^{t}
\end{equation}
Suppose now that $(t,n)=d>1$. It follows from here that there
exist $t_{1}, n_{1} \in N$ such that $t=dt_{1}$, $n=dn_{1}$ and
$(t_{1},n_{1})=1$. From \equref{fvsinf} we obtain
$g^{-1}a^{n_{1}}g=a^{nt_{1}}=1$, that is $a^{n_{1}}=1$, which is a
contradiction with $a$ having order $n$ and $n_{1}<n$. Hence
$(t,n)=1$ and $E\simeq <a \,,g \,| \, a^{n}=1 \, , \,
g^{-1}ag=a^{t}>$.

Now let $E\simeq <a\, ,g \, | \, a^{n}=1 \, , \, g^{-1}ag=a^{t}>$
for some $t\in\ZZ, (t,n)=1$. By \thref{4} we only need to prove
that $C_{n}\unlhd E$ and $E/C_{n}\simeq C_{g}$. For any $g' \in E$
we have $g'=x_{1}x_{2}...x_{k}$, for some $k \in \NN$, $x_{i} \in
\{a \, , \, g \, , \, a^{-1} \, , \, g^{-1}\}$, $i \in \{1 \, , \,
2 \, ,..., \, k\}$. That is, to prove that $C_{n}\unlhd E$ we only
need to show that $g^{-1}a^{l}g \in C_{n}$ and $ga^{l}g^{-1} \in
C_{n}$ for any $l \in \ZZ$. From \equref{fvsinf} we obtain, by
induction, that $g^{-1}a^{l}g=a^{tl} \in C_{n}$. Since $(t,n)=1$
there exist $\alpha, \beta \in \ZZ$ such that $\alpha t+\beta
n=1$. We obtain from \equref{fvsinf} that $a=ga^{t}g^{-1}$ and it
follows from here that $a^{\alpha}=ga^{\alpha t}g^{-1}$. Since
$ga^{\beta n}g^{-1}=1$ we obtain that $ga^{\alpha t+\beta
n}g^{-1}=a^{\alpha}$, that is $gag^{-1}=a^{\alpha}$. It follows
from here that $ga^{l}g^{-1}=a^{\alpha l}$ for any $l \in \ZZ$.
Hence $C_{n}\unlhd E$. It follows by a simple calculation that
every element $g' \in E$ can be written as $g^{p}a^{q}$ for some
$p, q \in \ZZ$. That is $gC_{n}=g^{p}a^{q}C_{n}=^{p}C_{n}$. Hence
$E/C_{n}\subseteq C_{g}$. Now suppose that there exist $\alpha,
\beta \in \ZZ$, $\alpha\neq \beta$ such that
$g^{\alpha}C_{n}=g^{\beta}C_{n}$, that is $g^{\alpha
-\beta}=a^{\gamma}$ for some $\gamma \in \{0 \, , \, 1 \, ,..., \,
n-1\}$. It follows from here that $g^{(\alpha- \beta)n}=a^{\gamma
n}=1$ which is a contradiction with $C_{g}$ being an infinite
cyclic group. Hence $E/C_{n}\simeq C_{g}$.
\end{proof}

\begin{theorem}\thlabel{infinitvsfinit}
A group E is isomorphic to a crossed product $C_{g}
\#_{\alpha}^{f} \, C_{n}$ if and only if :
\begin{enumerate}
\item[(i)]$E\simeq <g, \, h \, | \, gh = hg, \, h^{n} = g^{t},
t\in \ZZ>$ for n odd; \item[(ii)]$E\simeq <g, \, h \, | \, gh =
hg, h^{n} = g^{t}, t\in \ZZ>$ or $E\simeq <g,\, h \, | \, h^{n} =
1, ghg = h >$ for n even.
\end{enumerate}
\end{theorem}
\begin{proof}
 Suppose first that $E\simeq C_{g} \#_{\alpha}^{f}
\, C_{n}$. Hence $C_{g}\unlhd E$ and $E/C_{g}\simeq C_{n}$. That
is, there exists $h \in E$ such that $E/C_{g}=\{C_{g} \, , \,
hC_{g} \, ,..., \, h^{n-1}C_{g}\}$ and $h^{n} \in C_{g}$. Hence
there exists $t \in \ZZ$ such that $h^{n}=g^{t}$. Since
$C_{g}\unlhd E$ we obtain $h^{-1}gh \in C_{g}$, that is
$h^{-1}gh=g^{s}$ for some $s \in \ZZ$. It follows that
$h^{-1}g^{t}h=g^{st}$ and using $h^{n}=g^{t}$ we obtain
$g^{ts}=g^{t}$ that is $g^{t(s-1)}=1$. Since $C_{g}$ is a infinite
cyclic group we must have $t(s-1)=0$. Using again $h^{-1}gh=g^{s}$
we obtain $h^{-1}g^{s}h=g^{s^{2}}$, that is
$h^{-2}gh^{2}=g^{s^{2}}$ and by induction
$h^{-n}gh^{n}=g^{s^{n}}$. Thus, from $h^{n}=g^{t}$ we obtain
$g^{s^{n}-1}=1$ and since $C_{g}$ is an infinite cyclic group we
must have $s^{n}=1$. Therefore if $n$ is odd $E\simeq <g \, , \, h
\, | \, gh=hg \, , \,  h^{n}=g^{t}>$ for some $t\in \ZZ$ and if
$n$ is even $E\simeq <g \, , \, h \, | \, gh=hg \,, \,
h^{n}=g^{t}>$ for some $t\in \ZZ$ or $E\simeq <g \, , \, h \, | \,
h^{n}=1 \, , \, ghg=h>$.

We assume now that $E\simeq <g \, , \, h \, | \, gh=hg \, , \,
h^{n}=g^{t} \, , \, t\in \ZZ
>$. Since $E$ is abelian $C_{g}\unlhd E$. $E/C_{g}=\{g'C_{g}\, | \, g' \in
E\}$ and since every element $g' \in E$ can be written as
$g'=h^{p}g^{q}$ we obtain that
$g'C_{g}=h^{p}g^{q}C_{g}=h^{p}C_{g}$ that is $E/C_{g}\subseteq
\{C_{g} \, , \, hC_{g} \, , \,..., \, h^{n-1}C_{g}\}\simeq C_{n}$.
Now suppose that there exists $\alpha, \beta \in \{0 \, , \, 1 \,
,..., \, n-1\}$, $\alpha>\beta$, such that
$h^{\alpha}C_{g}=h^{\beta}C_{g}$ that is
$h^{\alpha-\beta}=g^{\gamma}$ for some $\gamma \in \ZZ$. Since
$\alpha-\beta < n$ we obtain a contradiction with $h^{n}=g^{t}$.
Hence $E/C_{g}\simeq C_{n}$ and by \thref{4} there exists
$(C_g,C_n,\alpha, f)$ a crossed system such that $E\simeq C_g
\#_{\alpha}^{f} \, C_n$.

Suppose now that $n$ is even and $E\simeq <g \, , \, h \, | \,
h^{n}=1 \, , \, ghg=h>$. By \thref{4} we only need to prove that
$C_{g}\unlhd E$ and $E/C_{g}\simeq C_{n}$. For any $g' \in E$ we
have $g'=x_{1}x_{2}...x_{k}$ for some $k \in \NN$, $x_{i} \in \{g
\, , \, h \, , \, g^{-1} \, , \, h^{-1}\}$ and $i \in \{1 \, , \,
2 \, ,..., \, k\}$. That is, to prove that $C_g\unlhd E$ we only
need to show that $hg^{l}h^{-1} \in C_{g}$ and $h^{-1}g^{l}h \in
C_{g}$ for any $l \in \ZZ$. Since $h^{-1}gh=g^{-1}$ we obtain, by
induction, that $h^{-1}g^{l}h=g^{-l} \in C_{g}$. Also from
$hg^{l}h^{-1}=h^{-n+1}g^{l}h^{n-1}=(h^{-1})^{n-1}g^{l}h^{n-1}=(h^{-1})^{n-2}g^{-l}h^{n-2}$
we obtain by induction $hg^{l}h^{-1}=g^{-l} \in C_{g}$. Hence
$C_{g}\unlhd E$ \,.$E/C_{g}=\{g'C_{g}\,|\,g' \in E \}$ and since
any element $g' \in E$ can be written as $h^{p}g^{q}$ for some
$p,q \in \ZZ$ it follows from here that
$g'C_{g}=h^{p}g^{q}C_{g}=h^{p}C_{g}$. Hence $E/C_{g}\subseteq
C_{n}$. Now suppose that there exist $\alpha, \beta \in
\{0\,,\,1\,,...,\,n-1\}$, $\alpha> \beta$ such that
$h^{\alpha}C_{g}=h^{\beta}C_{g}$, that is
$h^{\alpha-\beta}=g^{\gamma}$ for some $\gamma \in \ZZ$. It
follows from here that $g^{n\gamma}=(h^{n})^{\alpha-\beta}=1$ and
since $C_{g}$ is an infinite cyclic group we must have
$n\gamma=0$, that is $\gamma=0$. Hence $h^{\alpha-\beta}=1$ which
is a contradiction since the order of $h$ is $n$ and
$0<\alpha-\beta<n$. Therefore $E/C_{g}=C_{n}$.
\end{proof}

\begin{theorem}\thlabel{cicliceinfinite}
A group E is isomorphic to a crossed product $C_{g_{1}}
\#_{\alpha}^{f} \, C_{g_{2}}$ if and only if $E\simeq <g_{1},
g_{2} \, | \, g_{1}g_{2} = g_{2}g_{1} >$ or $E\simeq <g_{1}, g_{2}
\, | \, g_{1}g_{2}g_{1} = g_{2}>$.
\end{theorem}
\begin{proof}
Suppose first that $E\simeq C_{g_{1}} \#_{\alpha}^{f} \,
C_{g_{2}}$. Hence $C_{g_{1}}\unlhd E$ and $E/C_{g_{1}}\simeq
C_{g_{2}}$. That is $E/C_{g_{1}}=\{g_{2}^{k}C_{g_{1}}\,|\,k \in
\ZZ\}$ Since $C_{g_{1}}\unlhd E$ we obtain that
$g_{2}^{-1}g_{1}g_{2} \in C_{g_{1}}$ and $g_{2}g_{1}g_{2}^{-1} \in
C_{g_{1}}$. That is, there exist $s,t \in \ZZ$ such that
\begin{equation}\label{ciclicinf1}
g_{2}^{-1}g_{1}g_{2}=g_{1}^{t}
\end{equation}
and
\begin{equation}\label{ciclicinf2}
g_{2}g_{1}g_{2}^{-1}=g_{1}^{s}
\end{equation}
From (\ref{ciclicinf1}) we obtain, that
$g_{2}^{-1}g_{1}^{s}g_{2}=g_{1}^{st}$. It follows from here, using
(\ref{ciclicinf2}), that $g_{1}^{st}=g_{1}$. Since $C_{g_{1}}$ is
an infinite cyclic group, we obtain that $st=1$, that is $(s,t)
\in \{(1,1),(-1,-1)\}$. Hence $E\simeq
<g_{1},g_{2}\,|\,g_{1}g_{2}=g_{2}g_{1}>$ or $E\simeq
<g_{1},g_{2}\,|\,g_{1}g_{2}g_{1}=g_{2}>$.

Conversely, if $E\simeq <g_{1},g_{2}\,|\,g_{1}g_{2}=g_{2}g_{1}>$
then it is obvious that $E\simeq \ZZ\times\ZZ\simeq C_{g_1}\#
C_{g_2}$ the crossed system being the trivial one.

Now let $E\simeq <g_{1},g_{2}\,|\,g_{1}g_{2}g_{1}=g_{2}>$. Also by
\thref{4} we only need to prove that $C_{g_{1}}\unlhd E$ and
$E/C_{g_{1}}\simeq C_{g_{2}}$. For any $g \in E$ we have
$g=x_{1}x_{2}...x_{k}$ for $k \in N$, $x_{i} \in
\{g_{1}\,,\,g_{2}\,,\,g_{1}^{-1}\,,\,g_{2}^{-1}\}$ and $i \in
\{1\,,\,2\,,...,\,k\}$. Therefore, to prove $C_{g_{1}}\unlhd E$,
we only need to show that $g_{2}g_{1}^{l}g_{2}^{-1}\in C_{g_{1}}$
and $g_{2}^{-1}g_{1}^{l}g_{2} \in C_{g_{1}}$ for any $l \in \ZZ$.
From $g_{1}g_{2}g_{1}=g_{2}$ we obtain
$g_{2}g_{1}g_{2}^{-1}=g_{1}^{-1}=g_{2}^{-1}g_{1}g_{2}$ and
$g_{2}g_{1}^{l}g_{2}^{-1}=g_{1}^{-l}=g_{2}^{-1}g_{1}^{l}g_{2}$ for
any $l \in \ZZ$. Hence $C_{g_{1}}\unlhd E$. In a similar way it
can be shown that every element g of the group $E$ can be written
as $g_{2}^{p}g_{1}^{q}$ for some $p,q \in \ZZ$. It follows from
here that
$gC_{g_{1}}=g_{2}^{p}g_{1}^{q}C_{g_{1}}=g_{2}^{p}C_{g_{1}}$. Hence
$E/C_{g_{1}}\le C_{g_{2}}$. Since any non trivial subgroup of an
infinite cyclic group is infinite cyclic we obtain that
$E/C_{g_{1}}\simeq C_{g_{2}}$ which finishes the proof.
\end{proof}

\section{\textbf{When is a crossed product a cyclic group?}}\selabel{3}

Our aim in the present section is to give a necessary and
sufficient condition for a crossed product to be a cyclic group.
For this it is necessary that both groups should be cyclic since
any subgroup and any quotient of a cyclic group are cyclic groups.
Hence the problem is reduced to decide which of the crossed
products between two cyclic groups described in the previous
section are cyclic groups and under what conditions.

It is obvious that the crossed product between a finite cyclic
group $C_{n}$ and an infinite cyclic group ${C_{g}}$ described in
 \thref{finitvsinfinit} can not be a cyclic group since an
infinite cyclic group does not have torsion elements. By the same
argument we can conclude that the crossed product $<g\,,\,h\, |\,
h^{n}=1\,,\,ghg=h>$ obtained in \thref{infinitvsfinit} can not be
a cyclic group. Also, the crossed product between the two infinite
cyclic groups described in \thref{cicliceinfinite} can not be a
cyclic group because a nontrivial quotient of an infinite cyclic
group must be finite.

Therefore, the only crossed products left to deal with are : $C_n
\#_{i}^{j} \, C_m$ described in \thref{thmHolder} and $<g\,,\,h\,
|\, gh=hg\,,\, h^{n}=g^{t}\,,\,t\in \ZZ >$ from
\thref{infinitvsfinit}.

In what follows we investigate under which conditions these two
crossed products are cyclic groups.

In order to prove our next result we need the following technical
lemma:

\begin{lemma}
Let $m,n,i$ be rational integers so that $(m,n,i)=1$. Then there
exists some  $u,v,w\in\ZZ$ such that $um+vi+wn=1$ and $(m,v)=1$,
where $(r,s)$ denotes the greatest common divisor of the integers
$r$ and $s$.
\end{lemma}
\begin{proof}

Let $d=(m,n)$. Then $(d,i)=1$. Let $m'|m$ be so that  $(m',d)=1$
and $m'd$ contains all the prime factors of $m$. Using the Chinese
Reminder theorem we can find $v\in\ZZ$ such that $d|vi-1$ and
$m'|v-1$ (if $m'=1$ the last condition is trivially fulfilled). We
observe that $(m,v)=1$ because all the prime divisors of $m$ are
in $m'd$ and $(m'd,v)=1$.

Since $d=(m,n)$ there exist $u',w'\in\ZZ$ such that $u'm+w'n=d$.
From the way we chose $v$ it follows that there exists $r\in\ZZ$
such that $vi+rd=1$. Now put $u=ru', w=rw'$. From the above we
have $vi+r(u'm+w'n)=1$, thus $vi+um+wn=1$.
\end{proof}

\begin{proposition}\label{finvsfin=ciclic}
A crossed product $E = C_n \#_{i}^{j} \, C_m $ is a cyclic group
if and only if $j=1$ and $(m,n,i)=1$.
\end{proposition}

\begin{proof}

We know from \thref{thmHolder} that $E$ has a presentation of the
form $$E=<a\,,\,b\,|\,a^n=1\,,\, b^m=a^i\,,\, b^{-1} a b = a^j>.$$

Suppose first that $E$ is a cyclic group. It follows from here
that $j=1$ since every cyclic group is abelian. If $E$ is cyclic
then there exist some $u,v\in\ZZ$ such that $E=<a^ub^v>$. $a^ub^v$
has order $mn$, hence $(a^ub^v)^m$ has order $n$. It is well known
that in a cyclic group for any divisor of the order of the group
there exists a unique subgroup of that order, thus
$<(a^ub^v)^m>=<a>$ and therefore there exists some $k\in\ZZ$ such
that $(a^{uk}b^{vk})^m=a$. Using the relation $b^m=a^i$ and the
fact that $a$ has order $n$ we obtain that $ukm+vki-1$ is
divisible by $n$, that is $(m,n,i)=1$.

For the converse we will use the previous lemma and we obtain that
there exist $u,v,w\in\ZZ$ such that $um+iv+wn=1$ and $(m,v)=1$. We
will prove that $a^ub^v$ has order $mn$ in $E$ and that finishes
the proof. For this it is enough to prove that $a,b\in<a^ub^v>$.
By a simple calculation we get:
$(a^ub^v)^m=a^{um}a^{vi}=a^{um+vi}=a^{1-wn}=a$, that is
$a\in<a^ub^v>$. Since $(m,v)=1$, there exists $l\in\ZZ$ such that
$m|vl-w$. Now let $k=i+ln$.

Finally $(a^ub^v)^k=(a^ub^v)^i(a^ub^v)^{ln}=b^{um}b^{vi}b^{vln}=
b^{um+vi+vln}=b^{1-wn+vln}=bb^{n(vl-w)}=b$ because $n(vl-w)$ is
divisible by $mn$ and $|E|=mn$. Hence $b\in<a^ub^v>$.
\end{proof}

\begin{proposition}\label{inf-fin=ciclic}
The group $E = <g\,,\,h\,|\,h^n=g^t\,,\, hg=gh>$ where $n\ge
2,t\in\ZZ$ is cyclic if and only if $(n,t)=1$.
\end{proposition}
\begin{proof}
Denote by $d=(n,t)$ and by $Z_{n,t} = \ZZ+\displaystyle\frac tn\ZZ$ which is an abelian group and
is isomorphic to $(n,t)\ZZ\simeq\ZZ$ by the morphism $u\mapsto nu$.

Define $\theta:E\to Z_{n,t}$ by $h\mapsto \displaystyle\frac tn$ and $g\mapsto 1$. It is easy to
see that this is a morphism of groups and moreover it is surjective. In order to have $E$
isomorphic to $\ZZ$ (i.e. $E$ cyclic infinite), $\theta$ must be an isomorphism, otherwise we get a
surjective endomorphism of $\ZZ$ which is not injective and this is impossible.

So $E\simeq \ZZ$ iff $\theta$ is injective and this happens iff $(n,t)=1$. Indeed $h^rg^{-s}\mapsto
\displaystyle\frac{rt}n-s = 0\Leftrightarrow rt=ns\Leftrightarrow r=\displaystyle\frac{kn}d, s =
\displaystyle\frac{kt}d,\,k\in\ZZ$ so $\ker(\theta) = \{(h^{\frac nd}g^{-\frac td})^k: k\in\ZZ\}$
and then $\ker(\theta)=0$ iff $h^{\frac nd}=g^{\frac td}$, i.e. $d=1$.
\end{proof}

Our next goal is to describe, in the language of crossed systems,
all cyclic crossed products. That is, to identify the properties
that $(H,G,\alpha,f)$ has to verify in order to obtain a cyclic
crossed product $H\#_\alpha^f G$. As we already noticed, both $H$
and $G$ must be cyclic groups. Since $H\#_\alpha^f G$ is, in
particular, an abelian group it follows from \cite[Corrolary
1.15]{MA} that $\alpha$ must be trivial and $f$ a symmetric
2-cocycle. In order to find necessary and sufficient conditions on
$f$ such that $H\#^fG$ is cyclic we describe bellow all the
possible symmetric 2-cocycles.

For $m\ge 2$ and $n\ge 2$ or $n=\infty$ define $\Sigma_{m,n} = \{\phi:\ZZ\to\ZZ_n: \phi(0)=0,
\phi(t+m)=\phi(t),\forall t\in\ZZ\}$ with the convention that $\ZZ_\infty= \ZZ$.

\begin{proposition}\label{calcul_cociclii} The symmetric normalized 2-cocycles $f:C_m\times C_m\to C_n$ are in bijection with
the set $\Sigma_{m,n}$.
\end{proposition}
\begin{proof}
Let $\phi\in\Sigma_{m,n}$ and consider $x$ a generator of $C_m$ and $a$ a generator of $C_n$.

Denote, $S_k^\phi = S_k = \phi(0)+\ldots+\phi(k-1),\forall k\ge 1$.

We define $f(x^k,x^l) = a^{S_{k+l}-S_k-S_l}$ for $k,l\ge 1$ (observe that if $n\neq \infty$ then
$a^t$ is well defined for $t\in\ZZ_n$ since $a$ has order $n$).

It is easy to verify that $f(x^{k+sm},x^l) = f(x^k,x^{l+tm}) = f(x^k,x^l),(\forall) s,t\ge 0,
(\forall) l,k\ge1$. Observe also that $f(x,x^k) = a^{\phi(k)}$. This will be useful for the
converse.

$f$ is obviously symmetric.

We need to prove that $f$ is a 2-cocycle, that is:
\begin{eqnarray*}\label{2-cocycle}
f(x^k,x^l)f(x^{k+l},x^p) &=&f(x^l,x^p)f(x^k,x^{l+p})\\
a^{S_{k+l}-S_k-S_l+S_{k+l+p}-S_p-S_{k+l}} &=&
a^{S_{p+l}-S_p-S_l+S_{k+l+p}-S_k-S_{l+p}}\\
a^{S_{k+l+p}-S_k-S_l-S_p} &=& a^{S_{k+l+p}-S_k-S_l-S_p},(\forall)\, k,l,p\ge 1
\end{eqnarray*} and the later is clearly true.

So to each $\phi\in\Sigma_{m,n}$ we have associated a symmetric 2-cocycle.

Now suppose $f$ is a symmetric 2-cocycle. Define $\phi_f = \phi$ such that $a^{\phi(k)} =
f(x,x^k),k\in\ZZ$. It is obvious that $\phi\in \Sigma_{m,n}$ because $x$ has order $m$ and $f$ is
normalized.

Define $S_k^\phi = S_k = \phi_f(0)+\ldots+\phi_f(k-1)$.

Using the cocycle condition on $f$ and straightforward computation it follows that $f(x^l,x^k) =
a^{S_{k+l}-S_k-S_l},(\forall) k,l\ge 1$.

Hence, the map that associates to each cocycle the function $\phi_f$ is a bijective map between the
cocycles and $\Sigma_{m,n}$.
\end{proof}

\begin{proposition}\label{crossed}
A crossed product $C_n\#^f C_m, m,n\ge 2$ is a cyclic group if and
only if $(S_m,m,n) = 1$, where $S_k = \phi(0)+\ldots+\phi(k-1)$,
$\phi:\ZZ\to\ZZ_n, a^{\phi(k)} = f(x,x^k)$ and $x$ is a generator
for $C_{m}$.
\end{proposition}

\begin{remark}
Observe that $S_m$ is not a number, but a class (modulo $n$);
however $(S_m,m,n)$ does not depend on the choice of a
representant for $S_m$.
\end{remark}

\begin{proof}
We will prove that $C_n\#^f C_m$ is isomorphic to $C_n\#_i^1 C_m$
where $i\in\{0\,,\,1\,,\ldots,\,n-1\}$ such that $S_m=i$(mod $n$).
The conclusion will follow from Proposition \ref{finvsfin=ciclic}.

Let $i$ be the unique representant of $S_m$ from
$\{0\,,\,1\,,\ldots,\,n-1\}$.

Denote by $E = <a\,,\,b\,|\,a^n=1\,,\,b^m=a^i\,,\,ab=ba>$ and by
$F=C_n\#^f C_m$ the twisted product associated to the 2-cocycle
$f$ (see Example 1.2.3)

Define $\theta:E\to F$ by $\theta(a) = (a,1)$ and $\theta(b) =
(1,x)$.

 It is straightforward to see that $(1,x)^k =
(a^{S_k},x^k),\forall k\ge 1$ hence $(1,x)^m =
(a^{S_m},x^m)=(a^i,1)=(a,1)^i$ and $(1,x)^k \not\in <a>,\forall
k\in\{1\,,\ldots,\,m-1\}$. Also $(a,1)^n=1$ and
$(a,1)(1,x)=(a,x)=(1,x)(a,1)$. That is $(a,1)$ and $(1,x)$ verify
the same relations in $F$ as $a$ and $b$ do in $E$. Hence $\theta$
is a morphism of groups.

Let's observe that $(a,1)$ and $(1,x)$ generate the group $F$.
Indeed consider $(a^u,x^k)\in F$. Then $(a^u,x^k) = (a^u,1)(1,x^k)
= (a^{u-S_k},1)(a^{S_k},x^k) = (a,1)^{u-S_k}(1,x)^k$. Therefore
the morphism $\theta$ is also surjective and since the groups are
finite it is an isomorphism.
\end{proof}

\begin{proposition}\label{infinitvsfinit=ciclic}
A crossed product $C_g\#^f C_m$, $m\ge 2$, is cyclic iff $(S_m,m)=1$, where $S_m =
\phi(0)+\ldots+\phi(m-1),\phi:\ZZ\to\ZZ,g^{\phi(k)}=f(x,x^k),<x>=\ZZ_m$.
\end{proposition}
\begin{proof} We will prove that $C_g\#^f C_m$ is isomorphic to $E=<g\,,\,h\,|\,h^m=g^{S_m}\,,\,gh=hg>$ hence
the conclusion follows from Proposition \ref{inf-fin=ciclic}.

Denote by $F = C_g\#^f C_m$ the twisted product associated to the 2-cocycle $f$ (see Example
1.2.3).

Define $\theta:E\to F$ by $\theta(g) = (g,1)$ and $\theta(h) =
(1,x)$. It is easy to see that $(g,1)(1,x)=(g,x)=(1,x)(g,1)$ and
$(1,x)^k=(g^{S_k},x^k)$. Hence $(1,x)^m = (g^{S_m},1) =
(g,1)^{S_m}$. Therefore $\theta$ is a morphism of groups.
Moreover, since $(g^k,x^l) = (g,1)^{k-S_l}(1,x)^l$ we obtain that
$\theta$ is a surjection. Furthermore :
\begin{eqnarray*}
\ker(\theta) &=& \{g^kh^l: k,l\in\ZZ, \theta(g^kh^l)=(1,1)\}\\
 &=&\{g^kh^l:k,l\in\ZZ, (g,1)^k(1,x)^l=(1,1)\}\\
 &=&  \{g^kh^l:k,l\in\ZZ, (g^{k+S_l},x^l)=(1,1)\}\\
 &=&\{g^kh^l:k,l\in\ZZ,m|l, k=-S_l\}\\
 &=& \{g^kh^l:k,l\in\ZZ,l=sm,k=-sS_m\}\\
 &=& \{(g^{-S_m}h^m)^s:s\in\ZZ\}\\
 &=& \{1\}
\end{eqnarray*}
Hence $\theta$ is a bijection.

\end{proof}

In conclusion, with the above notations, we proved the following
theorem:

\begin{theorem}\thlabel{main_theorem}
A normalized crossed product $E=H\#_\alpha^f\, G$ is a cyclic group if and only if one of the
following are true:
\begin{enumerate}
\item $H\simeq C_n$, $G\simeq C_m$, for some $m,n\ge 2$, $\alpha$ is trivial and $(S_m,m,n)=1$
\item $H\simeq C_g$, $G\simeq C_m$, for some $m\ge 2$, $\alpha$ is trivial and $(S_m,m)=1$.
\end{enumerate}
\end{theorem}

Let us consider a numerical example. Define
$\phi\in\Sigma_{3,\infty}$ by $\phi(0)=0,\phi(1)=1,\phi(2)=1$ and
consider the corresponding symmetric 2-cocyle
$f:\ZZ_3\times\ZZ_3\to\ZZ$ (cf. Proposition
\ref{calcul_cociclii}).

An easy computation leads us to

$f(\hat0,\hat u) = f(\hat u,\hat0) = f(\hat2,\hat2) = 0,\,\forall\hat u\,\in\ZZ_3$

$f(\hat1,\hat1)=f(\hat1,\hat2)=f(\hat2,\hat1) = 1$

Since $S_3=2$ it follows from \thref{main_theorem} that $\ZZ\times^f\ZZ_3\simeq\ZZ$.

We can also find the generator of $\ZZ\times^f\ZZ_3$, namely $(0,\hat 2)$. Indeed:

$(0,\hat2)^2=(0+0+f(\hat2,\hat2),\hat2+\hat2) = (0,\hat1)$ and

$(0,\hat2)^3 = (0+0+f(\hat1,\hat2),\hat2+\hat1) = (1,\hat0)$.

\section*{Acknowledgement}

We wish to thank Professor Gigel Militaru, who suggested the
problem studied here, for his great support and for the useful
comments from which this manuscript has benefitted.


\begin{thebibliography}{99}

\bibitem{adem}
A. Adem, R. J. Milgram, Cohomology of finite groups, Springer, 2nd
Edition, 2004. Zbl 1061.20044

\bibitem{agore}
A.L. Agore, Constructions in group theory, dizertation, 2008,
Univ. of Bucharest (in romanian).

\bibitem{MA}
A.L. Agore, G. Militaru, {Crossed Product of Groups.
Applications}, {\sl Arabian J. Sci. and Engineering} {\bf
33}(2008), 1-17.

\bibitem{bechtell} H. Bechtell, The Theory of groups,
Addison-Wesley Publishing Company, 1971. Zbl. 0229.20001

\bibitem{grillet}
P. A. Grillet, Abstract Algebra, Graduate Texts in Mathematics
242, Springer, 2007. Zbl. 1122.00001

\bibitem{holder}
O. H\"{o}lder, {Bildung zusammengesetzter Gruppen}, {\sl Math.
Ann.} {\bf 46}(1895), 321--422.

\bibitem{rotman} J. Rotman, An introduction to the theory of groups. Fourth edition.
Graduate Texts in Mathematics 148, Springer-Verlag, New York,
1995. Zbl. 0810.20001

\bibitem{Weibel} C. Weibel, An introduction to homological algebra,
Cambridge Univ. Press, 1994. Zbl. 0797.18001


\end{thebibliography}
\end{document}